\documentclass{article}
\usepackage{times}  

\usepackage{soul}
\usepackage{url}
\usepackage[hidelinks]{hyperref}
\usepackage[utf8]{inputenc}
\usepackage[small]{caption}
\usepackage{bmpsize}
\usepackage{graphicx}
\usepackage{amsmath, amsthm}
\usepackage{enumitem}

\newtheorem{lemma}{Lemma}
\usepackage{booktabs}
\usepackage{subcaption}
\usepackage{multirow}
\usepackage{siunitx}
\usepackage{xcolor}
\usepackage{booktabs}
\usepackage{authblk}
    \newcommand{\midsepremove}{\aboverulesep = 0mm \belowrulesep = 0mm}
    \midsepremove
    \newcommand{\midsepdefault}{\aboverulesep = 0.605mm \belowrulesep = 0.984mm}
    \midsepdefault

\title{A Hybrid Framework Using a QUBO Solver For Permutation-Based Combinatorial Optimization}

\author{Siong Thye Goh}
\author{Sabrish Gopalakrishnan}
\author{Jianyuan Bo}
\author{Hoong Chuin Lau}

\affil{School of Computing and Information Systems, Singapore Management University, Singapore}

\begin{document}

\maketitle

\begin{abstract}
Hardware limitation imposes a challenge to apply quadratic unconstrained binary optimization (QUBO) / Ising model solvers directly to solve large combinatorial optimization problem instances.
In this paper, we propose a hybrid framework to solve large-scale permutation-based combinatorial problems effectively using a QUBO solver efficiently and effectively.  
To achieve this, several issues need to be addressed. First, we convert a constrained optimization model into an unconstrained model by introducing a penalty coefficient. Second, to ensure effective search of good quality solutions, a smooth energy landscape is needed; we propose a data scaling approach that reduces the amplitudes of the input without compromising optimality.  Third, parameter tuning is needed, and in this paper, we illustrate that for certain problems, it suffices to perform random sampling on the penalty coefficients to achieve good performance.  Fourth, we address the challenges of using a QUBO solver that may only work on small scale problems due to hardware limitation and also due to general performance deterioration of a QUBO solver as the size grows; we introduce a decomposition approach that calls the QUBO solver repeatedly on small sub-problems.  
Finally, to handle the possible infeasibility of the QUBO solution, we introduce a polynomial-time projection algorithm. We show our performance on the Traveling Salesman Problem and Flow Shop Scheduling Problem. We empirically compare the performance of our approach with solving the optimization problems directly.  We also compare the performance with an exact solver.
\end{abstract}

\section{Introduction}

Combinatorial optimization problems such as the Traveling Salesman Problem (TSP) and Flow Shop Scheduling Problem (FSP) 
are computationally intractable. Traditionally these problems are solved by modeling them as an integer program and solved with mathematical programming solvers such as CPLEX and Gurobi. These solvers are based on the branch and bound paradigm which are exponential time algorithms. 

Quantum computing offers a new approach to problem-solving. One exciting development is quantum annealing, which offers a general framework for solving combinatorial optimization problems efficiently and (provably) optimally, leveraging on quantum concepts like superposition and quantum tunneling (that enables a solution to escape from local optimality). 
While Quantum Annealing in its true form is still in the nascent stage of development, special quantum hardware such as D-Wave's quantum annealer has been developed that demonstrates effectiveness in solving targeted problems like Max-SAT and Max Cut (see e.g. \cite{venegas2018cross}).   Nurse-scheduling problem and flight gate assignment problem has been solved in \cite{ikeda2019application} and \cite{Flightannealing} respectively. In Quantum Annealing, a given combinatorial optimization problem is first formulated as an Ising model (or equivalently, quadratic unconstrained binary optimization problems (QUBO) \cite{lucas2014ising} and then solved using such annealing machine. Currently, these solvers assume that the given problem is unconstrained, and it has been proven in \cite{lucas2014ising} that a constrained problem such as those stated in a typical mixed integer programming (MIP) can be converted to an unconstrained problem using a penalty method. Such details will be covered in the background section.

Recently, we witness the development of QUBO solvers such as Alpha-QUBO \cite{glover2018tutorial} and Fujitsu's Digital Annealer (DA) \cite{aramon2019physics}, which can be executed on conventional machines without relying on quantum phenomena. Unlike D-Wave that requires the problem to be embedded in a Chimera graph (\cite{date2019efficiently} or (more recently) PEGASUS graph \cite{boothby2020next}, DA has the advantage that it works with a fully connected graph; in \cite{hamerly2018experimental}, the scaling advantage of all-to-all connectivity was discussed. 
DA is implemented as a special (but still CMOS-based) hardware processor that enables parallel exploration of the neighborhood search space.  It is based on a bit-flip simulated annealing algorithm with parallel tempering that does not require the user to specify the cooling schedule. Rather, the high and low temperatures are fixed and intermediate temperatures are adjusted with the objective of achieving an equal replica-exchange for all adjacent temperatures. It also includes a ``dynamic offset" mechanism that enables the algorithm to escape from local optimality \cite{aramon2019physics}. 

Unfortunately, all the above technologies suffer a limitation on the size of the model that can be solved directly, e.g. a limit of $2000$-qubits for D-Wave, $3000$ binary variables for Alpha-QUBO, and $8192$ binary variables for DA.
Even with prospective hardware enhancements, it is challenging to cope with large-scale combinatorial optimization problems as problem size increases. Hence, to date, it is commonly agreed that hybrid methods are needed. For instance, in \cite{liu2019modeling}, \cite{tran2016hybrid}, \cite{tran2016explorations}, and  \cite{venturelli2016job}, hybrid quantum-classical approaches to solving scheduling problems have been proposed. Furthermore, even if our problem size can fit into the solver input size, we show that divide and conquer can help us obtain better solution.

In this paper, we focus on permutation-based combinatorial optimization problems, i.e. problems whose constraints are permutation constraints. For a routing problem, this corresponds to deciding the order of nodes to visit; for a scheduling problem, this corresponds to deciding the order of tasks to be served by a machine.

It is known that these problems have dedicated solvers: for TSP,  Concorde has been designed to specifically solve it while for FSP, there are standard heuristics such as the NEH algorithm \cite{nawaz1983heuristic}. More complex neighborhoods for simulated annealing have also been proposed in  \cite{tian1999application}. 
Our goal is not to compete with these solvers nor specialized algorithms. Rather, we like to illustrate that a general QUBO solver that is based on simulated annealing with a simple bit-flip neighborhood can be exploited to solve such a problem effectively when incorporated within an algorithm framework that we propose in this paper. By doing so, we demonstrate that a classical generic SA-based QUBO solver can be exploited to solve large-scale permutation-based optimization problems efficiently on conventional computers. While this idea on its own is not novel, if we were to extrapolate this idea, we point the way forward that someday when a quantum annealing machine is more viable commercially, our proposed methodology will provide the backbone for solving large-scale combinatorial optimization problems with exponential speedup, thereby providing a strong competitor to commercial branch-and-bound-based exact solvers like CPLEX and Gurobi. 

With this backdrop (which is still some distance away), our contributions  are as follows: 
\begin{enumerate}
    \item We propose an algorithmic framework that decomposes a large problem instance into instances of a manageable size to achieve good performance rather than solving the QUBO directly.
    \item Under this framework, we propose a data scaling method to convert an instance to one with smaller cost variations while preserving the ranking of solutions for the original problem instance for QUBO with permutation constraint.
    We compare various approaches to tune the parameters for the smaller QUBO models. 
    We propose a method to project infeasible solutions obtained by the QUBO solver to feasible solutions using a polynomial-time weighted assignment algorithm.
   \item We apply our approach on Euclidean TSP  (E-TSP) and FSP instances and empirically evaluate our approach against different experimental settings. We show that we perform better than an exact solver for these problems.
\end{enumerate}

A noteworthy point is that there could be many engineering choices for performing divide and conquer when applied to specific permutation-based combinatorial optimization problems. In this paper, our goal is not to prescribe a one-size-fits-all algorithm for all such problems. Rather, we illustrate that it is quite straightforward to design a decomposition approach to deal with large scale problem instances.

\section{Background}

A general QUBO formulation can be written as follows:

\begin{equation}\min_{x \in \{0,1\}^n} x^TQx \label{qubo1}\end{equation}

\noindent where $Q$ is an $n \times n$ matrix and $x$ are binary decision variables.

Many combinatorial optimization problems (such as TSP, machine scheduling problems, FSP) can be formulated as quadratic binary optimization models with $m>0$ number of linear constraints:

\begin{equation} \min x^TQ_0x \end{equation}

\noindent subject to $c_i^Tx+d_i=0, \forall i \in \{1, \ldots, m\},$ which may be written as:

\begin{equation}\min x^TQ_0x \end{equation}

\noindent subject to $\|c_i^Tx+d_i\|^2=0, \forall i \in \{1, \ldots, m\}. $

Applying the penalty method, the constraints may be converted to the following form where $A_i$'s are the penalty terms. They introduce penalties to the objective function when the constraints to the original model are violated. It is known that when $A_i$'s are large enough, it is theoretically equivalent to the original constrained problem: 

\begin{equation} \min x^TQ_0x + \sum_i A_i\|c_i^Tx+d_i\|^2 \end{equation}

\noindent or more generally:

\begin{equation}\min x^TQ_0x + \sum_i A_i(x^TQ_ix+\alpha_i) \label{generalqubo}\end{equation}

\noindent where $Q_i$ is an $n \times n$ matrix and $\alpha_i$ are constants. We require $x^TQ_ix+\alpha_i \ge 0$.

QUBO is closely related to the Ising model in Statistical Mechanics. Rather than dealing with binary variables, the latter works with spins (variables) which take values from $\{-1, 1\}$ and the goal is to reach the minimum state of a Hamiltonian energy function $H$ via a process known as quantum annealing. An Ising model can be directly and linearly transformed into a QUBO model \cite{lucas2014ising}.

As mentioned earlier, theoretically when $A_i$ is sufficiently large, the model is equivalent to the original constrained problem. In practice, however, when a heuristic solver such as DA or Alpha-QUBO is applied to solve a QUBO model, its performance is sensitive to the values of $A_i$, as they determine the landscape of the objective function. Figure \ref{improveeffect} illustrates that parameter tuning can improve the performance of a QUBO solver for E-TSP. This is especially true if the QUBO size is large and if there is no pre-processing. Unfortunately, the tuning process is rarely discussed in the literature in the context of problem model parameters, vis-a-viz automated tuning or configuration of algorithm-specific parameters.  Typically, high-level guidance such as those presented in \cite{lucas2014ising} is used as a norm. Some exceptions are in  \cite{karimi2017subgradient}, where the weight for a single parameter is tuned using a sub-gradient approach to solve the quadratic stable set problem and \cite{huang2021qross} where  QROSS method, in which surrogate models of QUBO solvers are being built to reduce the number of calls to the QUBO solver is being made.

\begin{figure}[htp]
\center
\includegraphics[width=8cm]{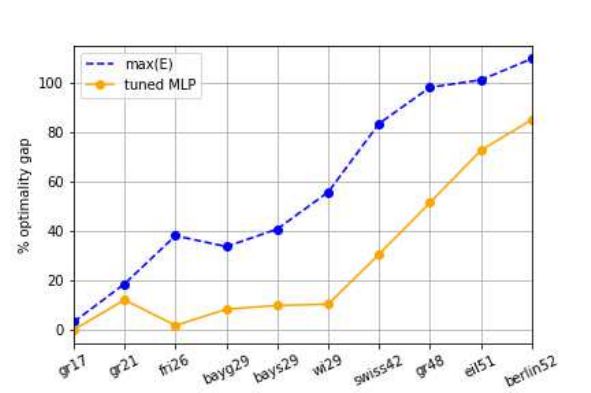}
\caption{Effect of tuning on optimality gap for Traveling Salesman Problem. It is compared to the baseline (in blue) where the penalty coefficient is chosen to be maximum edge length as this is a theoretical bound that suffices to make the QUBO equivalent to the original TSP.}\label{improveeffect}
\end{figure}

\subsection{Permutation-based Problems}
A permutation-based combinatorial optimization problem involves permuting $n$ objects to minimize a certain objective function. Common examples of permutation-based problems include the traveling salesman problem, the permutation flow shop scheduling problem, and the quadratic assignment problem. Such problems can be modeled as minimizing a quadratic objective function of the following form: 
\begin{equation}\min \sum_{i=1}^n\sum_{j=1}^n \sum_{u=1}^n \sum_{v=1}^n x_{u, i}Q_{u,i,v, j}x_{v, j}\label{originalqubo}\end{equation}

\noindent subject to $\sum_{u=1}^n x_{u,i}=1, \forall i \in \{ 1, \ldots, n\}$
and
$\sum_{i=1}^n x_{u,i}=1, \forall u \in \{1, \ldots, n\}$
 where $x_{u,i}$ takes value $1$ if object $u$ is assigned to slot $i$ and it takes value $0$ otherwise. The two constraints ensure that each object is given a slot and vice versa. We call the first group of constraint the column sum constraints and the second group of constraint the row sum constraints.

We can convert this optimization formulation to a QUBO model by squaring the constraint violations and add it to the original objective function:

\begin{align}\min &\sum_{i=1}^n\sum_{j=1}^n \sum_{u=1}^n \sum_{v=1}^n x_{u,i}Q_{u,i, v, j}x_{v,j} \\&+ A \left[ \sum_{u=1}^n \left(\sum_{i=1}^n x_{u,i}-1 \right)^2 +\sum_{i=1}^n \left(\sum_{u=1}^n x_{u,i}-1 \right)^2  \right] \label{tune}\end{align}

\noindent where $A$ is a single parameter that we have to tune to enforce the constraint. Henceforth, we call a QUBO expressed in that form a "permutation QUBO".

For this work, we choose to tune a single parameter $A$ following the formulation given in \cite{lucas2014ising} although we could have assigned a new parameter to each of the $2n$ constraints. By doing so, the number of parameters that we have to tune does not grow as $n$ increases. The symmetry in the permutation constraints suggests that we do not have to emphasize some constraints over the other intuitively.

To enhance the performance of the QUBO solver, we propose the following ideas. 
First, we propose a data scaling procedure to scale the objective function so as to make the objective landscape smoother to improve the search effectiveness. Second, we take advantage of the permutation structure and introduce a method to ensure that the solution obtained from a QUBO solver can be easily projected to a feasible solution by solving a weighted assignment problem (which can be solved in polynomial time). 

\section{Solution Approach}

Given the above motivation, we propose our algorithm framework summarized in Figure \ref{bigpicture} to solve large scale permutation-based combinatorial problems. Details are described in the following sub-sections.

Note that our approach is applicable to any QUBO solver. 
Henceforth, we will use the term \textbf{QS} to denote the QUBO Solver. 

\begin{figure}
\includegraphics[width=9cm]{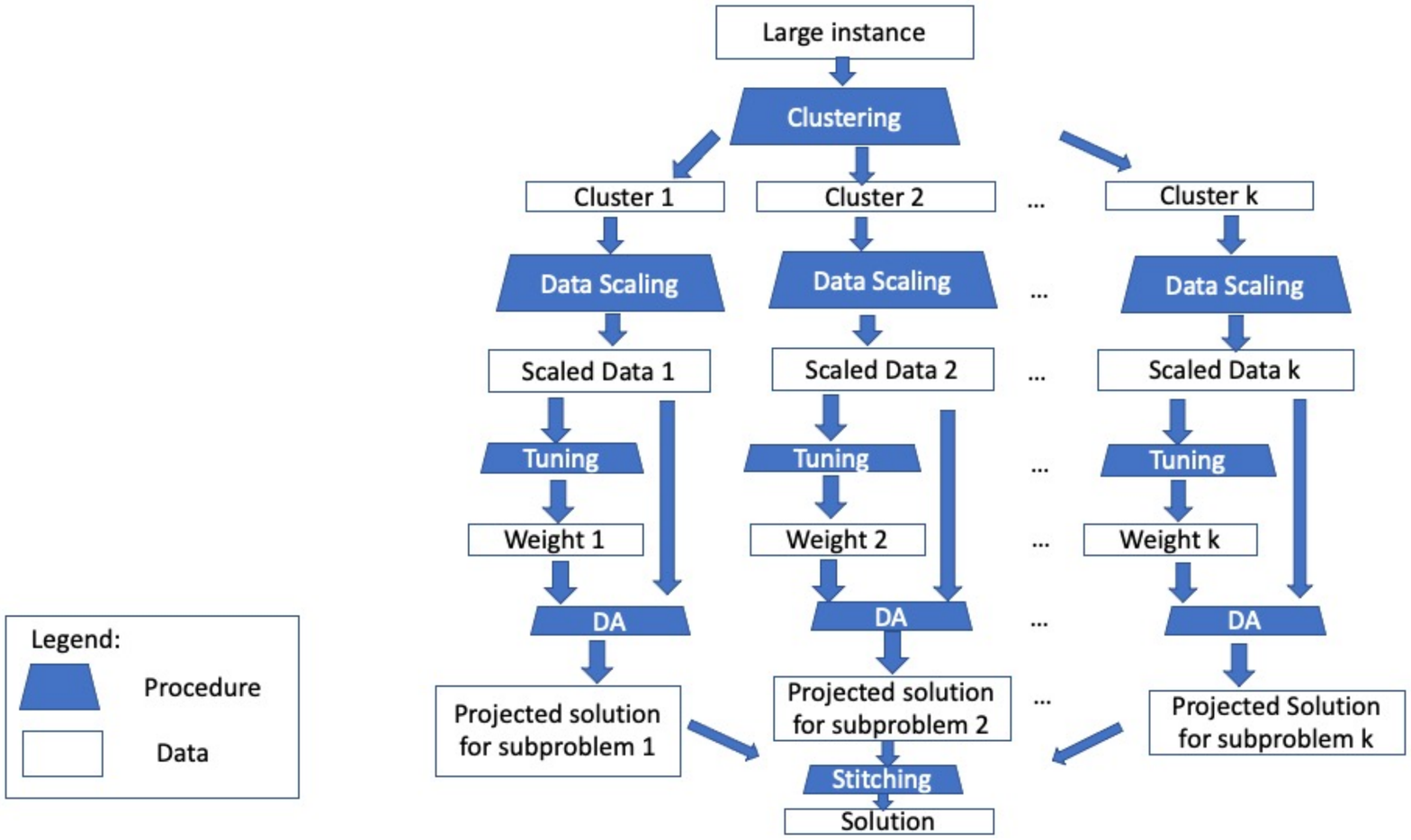}
\caption{Proposed algorithmic framework  }\label{bigpicture}
\end{figure}

\subsection{Clustering}


Although the concept of decomposition through clustering is generic, it can be made to be more effective by taking advantage of the problem structure. For instance, if the objective is to minimize distance, we can make use of efficient procedures such as $k$-means clustering. Otherwise, a more general procedure such as spectral clustering \cite{ng2002spectral} can be used. 

Furthermore, for permutation-based problems, clustering (and subsequent stitching) is natural, since a permutation can be divided into sub-permutations which can be readily combined without the fear of losing feasibility. 

Note that typically the number of clusters formed is a function of the size of the subproblem within each cluster that QS performs well in terms of solution quality and solving time. The number of clusters affects the size of each sub-problem, which should not be too small as that would affect the performance when we merge the solutions of sub-problems to form a solution to the original problem. In this paper, we will use the notation $k$ to denote the number of clusters.


\subsection{Data Scaling for Each Cluster}

In \cite{held1970traveling}, it has been shown that for the TSP problem, one can reduce all the distances to or from a particular city by a constant and preserve the ranking of the solutions.  This result can be extended to the quadratic binary optimization problem with permutation constraints as follows.

Given the optimization problem ($\ref{originalqubo}$), pick any $\hat{j} \in \{1, \ldots, n\}$ and define $\tilde{Q}$ as follows:

$$\forall u, i, v \in \{1, \ldots, n\}, \tilde{Q}_{u, i, v, j} = \begin{cases} Q_{u,i, v, j} & \text{if } j \ne \hat{j} \\ Q_{u,i,v,j} +\Delta & \text{if } j= \hat{j}\end{cases}$$

That is, we change those matrix entries with index $j = \hat{j}$ by a constant $\Delta$. We call this process data scaling.

Consider the resulting scaled optimization problem:

\begin{equation}\label{perturbqubo}\min \sum_{i=1}^{n}\sum_{j=1}^n \sum_{u=1}^n \sum_{v=1}^n x_{u,i}\tilde{Q}_{u,i,j,v}x_{j,v} \end{equation}

\noindent subject to 
$\sum_{u \in I} x_{u,i}=1$ and
$\sum_{i \in J} x_{u,i}=1$.

The following lemma shows that the ranking of optimality is preserved under data scaling: 

\begin{lemma}
A solution $y$ that is better than $z$ for optimization problem $(\ref{originalqubo})$ remains better for optimization problem $(\ref{perturbqubo})$.
\end{lemma}

Note that our theoretical result can be further generalized to constraint of the form of $\sum_{i \in I}x_{i,j}=B$ and $\sum_{i \in I}x_{i,j}=C$ where $B$ and $C$ satisfy $B|J|=C|I|$. We state and prove the more general version of the result here:

Consider the following quadratic programming problem with binary variables.

\begin{equation}\label{originalobj}\min \sum_{\substack{i,u \in I\\ v, j \in J}} x_{i,j}Q_{i,j,u,v}x_{u,v}  \end{equation}

\noindent subject to 
$\sum_{i \in I} x_{i,j}=B$ and
$\sum_{j \in J} x_{i,j}=C$
where $|J|B = |I|C$.
That is we impose the conditions that the row sum is a constant and the column sum is another constant and the problem is feasible.

Let $\hat{j} \in J$.
Consider the optimization problem with perturbed data:

\begin{equation}\label{perturbobj}\min \sum_{\substack{i,u \in I\\ v, j \in J}} x_{i,j}\tilde{Q}_{i,j,u,v}x_{u,v} \end{equation}

\noindent subject to 
$\sum_{i \in I} x_{i,j}=B$ and
$\sum_{j \in J} x_{i,j}=C$.

where

$$\tilde{Q}_{i,j,u,v} = \begin{cases} Q_{i,j,u,v} & \text{if } j \ne \hat{j} \\ Q_{i,j,u,v} - \Delta & {if } j= \hat{j}\end{cases}$$

The following lemma shows that the ranking of optimality is preserved under data scaling: 

\begin{lemma}
A solution $y$ that is better than $z$ for optimization problem $(\ref{originalobj})$ remains better for optimization problem $(\ref{perturbobj})$.
\end{lemma}
  
\begin{proof}
It suffices to show that the difference in the objective values of the two solutions $y$ and $z$ remain the same before and after scaling. More precisely, 
$\sum_{\substack{i,u \in I\\ v, j \in J}} z_{i,j}Q_{i,j,u,v}z_{u,v} -  \sum_{\substack{i,u \in I\\ v, j \in J}} y_{i,j}Q_{i,j,u,v}y_{u,v}=\sum_{\substack{i,u \in I\\ v, j \in J}} z_{i,j}\tilde{Q}_{i,j,u,v}z_{u,v} -  \sum_{\substack{i,u \in I\\ v, j \in J}} y_{i,j}\tilde{Q}_{i,j,u,v}y_{u,v}$.
 
By definition we have,
\begin{align*}
&\sum_{\substack{i,u \in I\\ j,v \in J} } z_{i,j}\tilde{Q}_{i,j,u,v} z_{u,v} - y_{i,j} \tilde{Q}_{i,j,u,v} y_{u,v} \\
&= \sum_{\substack{i,u \in I \\ j,v \in J}} [z_{i,j}Q_{i,j,u,v} z_{u,v} - y_{i,j}Q_{i,j,u,v} y_{u,v}] \\ \end{align*}
\begin{align}
&-\Delta \sum_{\substack{i,u \in I} ,v \in J} [z_{i,\hat{j}} z_{u,v} - y_{i,\hat{j}} y_{u,v}] \label{secondterm}
\end{align}

Since both solutions are feasible, i.e. 
$\sum_{i \in I} y_{i,j}=\sum_{i \in I} z_{i,j}=B$ and 
$\sum_{j \in J} y_{i,j}=\sum_{j \in J} z_{i,j}=C$,  

\begin{align*}
\sum_{\substack{i,u \in I \\ v \in J}} y_{i, \hat{j}}y_{u,v} &= \sum_{i \in I} y_{i ,\hat{j}} \sum_{u \in I} \sum_{v \in J} y_{u,v} \\
&= B \sum_{u \in I} C \\
&= BC|I|
\end{align*}
and similarly
$
\sum_{\substack{i,u \in I \\ v \in J}} z_{i, \hat{j}}z_{u,v}
= BC|I|
$

Hence the term $(\ref{secondterm})$ is equal to 0. 


\end{proof}

Specifically when $B=C=1$, and the index set $I=J=\{ 1, \ldots, n\}$, we have the special case for permutation based optimization problems which is Lemma $1$. 

Even though we have shown that we can change the objective value corresponding to certain indices by a constant and yet the two problems remain equivalent, it is interesting to determine the value to be scaled for a particular index $i$, i.e. $\Delta_i$. One approach is to use the result from \cite{wang2018distance} to minimize the variance of all the scaled distances. Intuitively, this corresponds to smoothing the landscape of the objective function.

\subsection{Parameter Tuning}

As discussed above, it is important to tune the parameter $A$ in Equation \ref{tune}  unless the search space  can be narrowed such that the values within it does not impact the objective function.


 For a general QS, there are several approaches for finding  suitable parameter values: 

 \textbf{Offline training:} One approach is to develop a a regression model use a multi-layer perceptron (MLP) that will predict a suitable weight parameter value for a given instance.  
 We first perform exploration by sampling the penalty coefficient values over a wide range and then sampling more values over a narrower range around the parameter values that have better performance. For TSP for example,  one can a neural network to predict $\frac{\text{Suitable weight}}{\text{Longest edge length}}$. We also encode our QUBO into common graph properties such as maximum and minimum edge weight, the minimal spanning tree, condition number of the graph to be used as the input features of the MLP.  Yet another approach of offline training is  \cite{huang2021qross} where we build  a surrogate model 
 of the QS via learning from solver data on a collection of problem instances, and using it to search for a suitable penalty parameter. For this manuscript, we have excluded offline training experimental results in the numerical section since the training process is very time consuming and similar performance can be attained by the online approaches (which we will discuss below).

\textbf{Online parameter search:} One approach is to perform an online search for suitable parameters. Bayesian approaches are model-based approaches that updates the search regions dynamically. Hyperopt \cite{bergstra2013hyperopt} and Optuna\cite{akiba2019optuna} are frequently used to improve the quality of the solution iteratively. Model-free approaches based on Particle Swarm Optimization (PSO) \cite{rini2011particle} are also able to achieve  reasonable results. Furthermore, while there are multiple hyper-parameters that can be set for PSO, interestingly we notice that the performance is not sensitive to them for our problems.

\textbf{Statistical sampling:} Another approach is to draw the parameter values based on some distributions.  For example, we can collect data to learn the average value and standard deviation of penalty parameter of the problem instance  that performs well and fit a distribution to it. During the prediction stage, we draw samples from the fitted distribution repeatedly.


The choice of which parameter tuning scheme to use depends largely on the applications, and in the experiments that follow, we compare the results of these methods. 


\subsection{Projection to the Feasible Space}

While ideally, QS should return feasible (though not optimal) solutions, this need not always be the case. In this paper, we propose a projection algorithm to map an infeasible solution to a feasible solution.

Suppose $z\in \{0,1\}^{n \times n}$  is an infeasible solution returned by QS. To restore feasibility of the original constrained problem, we solve the following optimization problem:

\begin{align*}&\min \sum_{i=1}^n\sum_{j=1}^n (x_{i,j}-z_{i,j})^2\\&= \min \sum_{i=1}^n\sum_{j=1}^n ((1-2z_{i,j})x_{i,j}+z_{i,j})\end{align*}

\noindent subject to $\sum_{i=1}^n x_{i,j}=1 , \forall j \in \{1,\ldots,n\}$ and $\sum_{j=1}^n x_{i,j}=1, \forall i \in \{1,\ldots,n\}$.

Note that here $z$ would be a given constant and hence this reduces to the standard Weighted Assignment Problem which can be solved in $O(\left(\frac{n}{k}\right)^3)$ time with the Hungarian algorithm \cite{jonker1987shortest}.

\subsection{Stitching}

Having solved the sub-problems, stitching involves combining the solutions obtained to construct a feasible solution to the original problem. 

A nice property of permutation-based problems is that any permutation of $\{1, \ldots, n\}$ is a feasible solution. Hence finding good feasible solution hinges on finding an effective scheme to stitch the solutions of the clusters together (assuming that we do not need to move items between the clusters).  
If the number of clusters $k$ is small, we can enumerate the configurations to find the best configuration. Otherwise, we can view the clusters themselves as items for the macro problem which we can apply our approach recursively, where the distance between the clusters would be problem-specific. For example, for TSP, we can define the distance between two clusters to be the smallest or largest distance of a city in the first cluster with a city in the second cluster.

In general, the cost of each sub-problem is a function of $\frac{n}{k}$, which is the size of each sub-problem. We have to trade-off between the speed of solving each sub-problem and the total time taken to stitch all the sub-problems together. The details of the complexity depends on the choice of QS and the stitching procedure.



\hspace{10mm}



\section{Applications}

Next, we will illustrate the application of our solution approach on the Euclidean Traveling Salesman Problem (E-TSP) and Flow Shop Scheduling Problem (FSP). 

\subsection{Application I: Euclidean Traveling Salesman Problem (E-TSP)}

While our framework works for arbitrary TSP instances, we focus on E-TSP due to the famous constrained $k$-means algorithm \cite{bradley2000constrained} can be applied easily.

In \cite{lucas2014ising}, a QUBO formulation that only involves a quadratic number of terms in the number of cities is proposed. Without loss of generality, we can focus on the case where the graph is fully connected, as we can always introduce edges of infinite distances otherwise. We let $d_{uv}$ be the distance between city $u$ and city $v$. We require $n^2$ variables for an $n$-city instance. The first subscript of $x$ represents the city and the second indicates the order that the city is going to be visited at. That is $x_{v,j}$ is the indicator variable that the city $v$ is the $j$-th city to be visited. Notice that the constraint implies that this satisfies the permutation condition.

The formulation is as follows: $\min_{x} H_B(x) + AH_A(x)$.

Here \begin{equation*}H_B(x)=\sum_{(u,v) \in E} d_{uv} \sum_{j=1}^n x_{u,j}x_{v,j+1}\end{equation*} describes the total distance travelled and 

\begin{equation}H_A = \sum_{v=1}^n \left( 1-\sum_{j=1}^n x_{v,j}\right)^2 + \sum_{j=1}^n \left( 1-\sum_{v=1}^n x_{v,j}\right)^2 \end{equation} describes the constraints to be a feasible cycle.

We are interested in focusing on decomposing our problems into instances of a certain size, hence as stated earlier in the clustering stage, we use the constrained $k$-means clustering algorithm \cite{bradley2000constrained}. The benefit of this approach is that we have better control over the size of the clusters suitable for QS. 

After clustering, we use the variance-minimizing distance scaling to change the objective to make the objective landscape smoother.

Let $k$ be the number of clusters and largest cluster size  be $|V_c|$. The stitching algorithm is outlined as follows: 

\begin{enumerate}
\item  We define the least cost flip value $\Delta_{ij}$ between all cluster pairs to be the least cost of performing 2-opt to stitch the two clusters $i$ and $j$ together. The time complexity is  $ O(k^2{|V_c|}^2)$.
\item To determine the ordering of stitching clusters, we solve a minimum cost Hamiltonian path problem by reusing our QUBO E-TSP model presented above (except it seeks a minimum Hamiltonian path \cite{lucas2014ising} instead of cycle) with $\Delta_{ij}$ on the edges.  
\item Finally, we construct the final solution for the original problem by stitching adjacent clusters derived from Step 2 via 2-opt. The stitching scheme is illustrated in Figure \ref{stitch}.
\begin{figure}[htbp]
\begin{center}
\includegraphics[width=8cm]{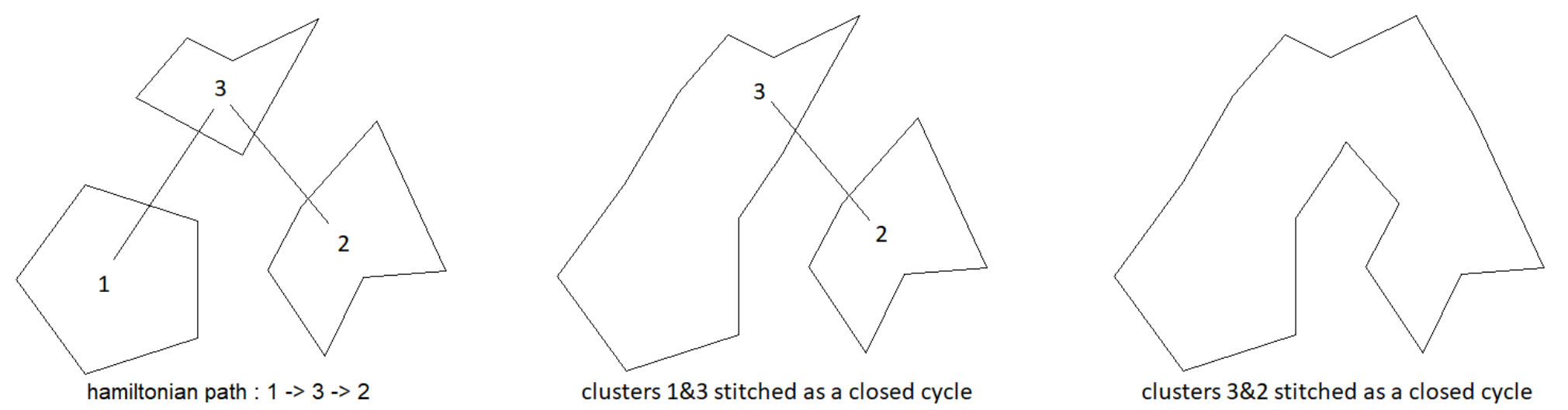}
\caption{Stitching procedure}\label{stitch}
\end{center}
\end{figure}

\end{enumerate}


\subsection{Application II: Flow Shop Scheduling Problem (FSP)}


Given $n$ jobs and $m$ machines, and each job $j, j\in \{1,2,\ldots,n\}$ has exactly $m$ operations where the $i^{th}$ operation is to be performed in order by machine $i$. In other words, each job $j$ is first processed by machine $1$, then machine $2$, and so on, until it completes its processing by machine $m$.  The time required for each job $j$ on machine $i$, $t_{ij}$ is given to us. The time at which job $j$ completes its last operation shall be denoted by $C_j$. Each machine can only work on a single operation at any point of time. The objective is to find a non-preemptive schedule that minimizes the maximum completion time of any job by obtaining the optimal permutation of the jobs.  We want to minimize $C_{\max} = \max_j C_j$, 
where $C_j$ denotes the time at which job $j$ completes its last operation $m$. Let the optimum value for a given instance be denoted as $C^*_{\max}$.

To use a QUBO solver to solve the permutation FSP, we apply the approximation proposed by Gupta\cite{gupta1968general} and has been studied by Stinson and Smith \cite{stinson1982heuristic}, Widmer and Hertz \cite{widmer1989new}, and Moccellin \cite{moccellin1995new}. 
In these papers, the original $n$ jobs FSP problem is transformed into a TSP problem of $n$ cities with the distance between two jobs defined by various formulation mostly involving the processing time of these two jobs.
Unlike the regular Euclidean TSP that we have discussed earlier, we no longer have the coordinate of the cities but we just have the distances and they need not be Euclidean. Once the TSP instance is solved, the solution obtained is converted to the solution of the original FSP.

Among all different transformation formulations, we choose the distance formulation named ``Sum of the Absolute Residuals with Negative Residuals Weighted Double - no carryover" in \cite{stinson1982heuristic} based on the experiment results. A different clustering approach is needed. Spectral clustering \cite{ng2002spectral} techniques make use of the spectrum (eigenvalues) of the similarity matrix of the data to perform dimensionality reduction before clustering in fewer dimensions. The similarity matrix is provided as an input and consists of a quantitative assessment of the relative similarity of each pair of points in the data set. We use spectral clustering to cluster the different jobs into clusters.

 To construct the similarity matrix from the distance matrix, we use the following conversion $Similarity_{i,j} = \exp \left( \frac{-Distance_{ij}^2}{2\delta^2}\right).$ $\delta$ here is the scaling factor which is chosen to be the average value of the distance matrix from the experiment results. After transformation, we could have all the elements in the similarity matrix to be in the range $[0,1]$ and a higher value suggests that two jobs are more similar to each other.

 After decomposition, the smaller scale problem will be solved using the same QUBO formulation mentioned previously. In other words, each smaller scale problem derived after the decomposition will be treated as an independent FSP and sent to QS simultaneously to be solved. After the solutions have been obtained from QS, a permutation of the sequence of each cluster will be conducted and the sequence with the minimum total makespan of the original problem is obtained by a brute-force search. Finally, for this problem,  we permute the group of the ordering of the job clusters to stitch the job together and select the best cluster permutation.

 Such scheduling problem in practice tend to be not too large,  we permute the group of the ordering of the job clusters to stitch the job together and select the best cluster permutation.

\section{Numerical Experiments}

In this section, we present experimental results that
\begin{enumerate}[label=(\Alph*)]
\item investigate the effect of data scaling on the quality of the solution; 
\item investigate the performance when we use a QS directly and do not use a hybrid framework.
\item investigate the performance when we use an exact solver, CPLEX directly and not using a QS; and
\item  investigate the performance when we apply our framework on E-TSP based on different parameter tuning approaches.
\item investigate the performance when we apply our framework on FSP based on different parameter tuning approaches.
\end{enumerate}

In the following we will present our numerical results based on the Fujitsu Digital Annealer that runs in the Parallel Tempering mode as the QUBO solver (QS).

We explain the experimental setup below. 

For clustering, we use the constrained $k$-means algorithm  \cite{bradley2000constrained} to control the size of each sub-QUBO.
The minimal cluster size, $\tau$, is set to be $7$ and the maximum cluster size, $\mu$ is fixed at $30$.

In terms of parameter tuning, for the online parameter search approach, we evaluate $40$ parameters in one run for each tuning method.
For the two Bayesian approaches Hyperopt and Optuna, we set $[0.5, 1] \cdot D_{max}$ as the range of searching space with 40 trials.
For the model-free PSO approach, the position of the particle represents the parameter to be tuned. The cost function of the PSO is set to be the objective function according to the application we are solving. Besides, we need to tune several hyperparameters for PSO including the inertia weight $\omega$, two learning factors $c_1$ and $c_2$. $\omega$ controls the contribution rate of the particle's previous velocity to the particle's velocity at the current time step. A combination of $c_1$ and $c_2$ determines the convergence rate by balancing the global and local search capabilities. Based on the experiment results, we set the inertia weight $\omega_{\max} = 0.5$, and $\omega_{\min} = 0.25$. The learning factor $c_{1,\max}=0.5$, $c_{1, \min}=0.25$, $c_{2, \max}=0.9$, and $c_{3, \max}=0.6$.
We update $\omega$, $c_1$, and $c_2$ according to the linear time varying formulation in \cite{785514}. In order to have total $40$ parameters evaluated, PSO runs $4$ iterations (including the random initialization) with number of particles equals to $10$. The search space are same as the Bayesian approaches. For the offline parameter search, we built a multi-layer perceptron, however, since the effect is similar to the online version, we do not present the result here.  

For the two statistical sampling approaches, we apply normal distribution and uniform distribution respectively. The normal distribution is $\mathcal{N}(\mu = 0.7594, \sigma^2 = 0.0141) \cdot D_{\max}$, where $D_{\max}$ is the largest magnitude of the distance matrix. $\mu$ and $\sigma$ are estimated by repeatedly calling the solvers and recording down parameter values that perform well. For uniform distribution, we sample from $\mathcal{U}(0.5, 1)\cdot D_{max}$ as the baseline method. Both the normal and uniform distributions are sampled $40$ times each.

\subsection{Effect of Data Scaling}

Data scaling reduces the ruggedness of the energy landscape. In Figure \ref{improve}, we illustrate that data scaling improves the performance of TSP in that the solutions obtained will be strictly better than those without perturbation.  Notice that the same data scaling scheme is used for FSP as well.

\begin{figure*}[!ht]
\includegraphics[width=\textwidth]{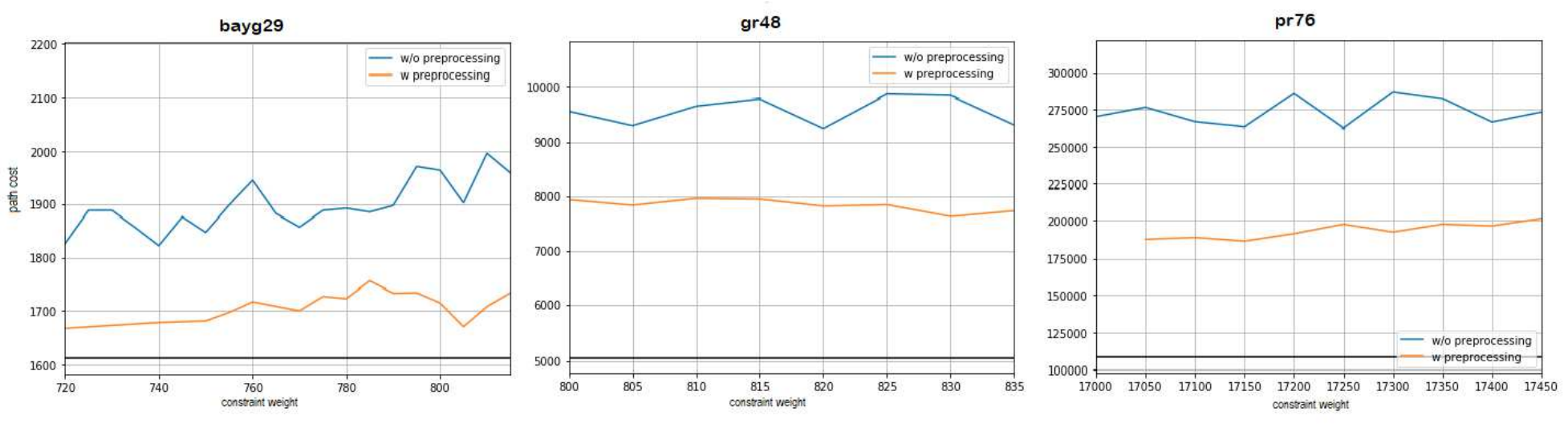}
\caption{Improvement by performing data scaling}\label{improve}
\end{figure*}

\subsection{Comparison with using QS Directly on Small Instances} 

Due to hardware limitations, QS may not be used directly when the problem instances are large, which entails a decomposition approach such as ours. That said, in this section, we will empirically evaluate the performance of our hybrid framework against directly solving small QUBO instances with QS directly.  

In Figure \ref{fig:directvs}, we illustrate the difference on solving TSP instances of up to $90$ cities from two data repositories. Using the same computational budget, we observe that the quality of solution of our approach is better: typically, our hybrid approach obtains an optimality gap that is half of the approach where we directly use QS. We show that when the instance size is more than $50$ cities for TSP, the optimality gap can be more than $35\%$ without using our framework (as it can fit into the hardware limitations of QS).  In contrast, the optimality gap for our framework is less than $10\%$.

\begin{figure}
\centering
\begin{subfigure}{0.45\textwidth}
  \centering
  \includegraphics[width=0.9\linewidth]{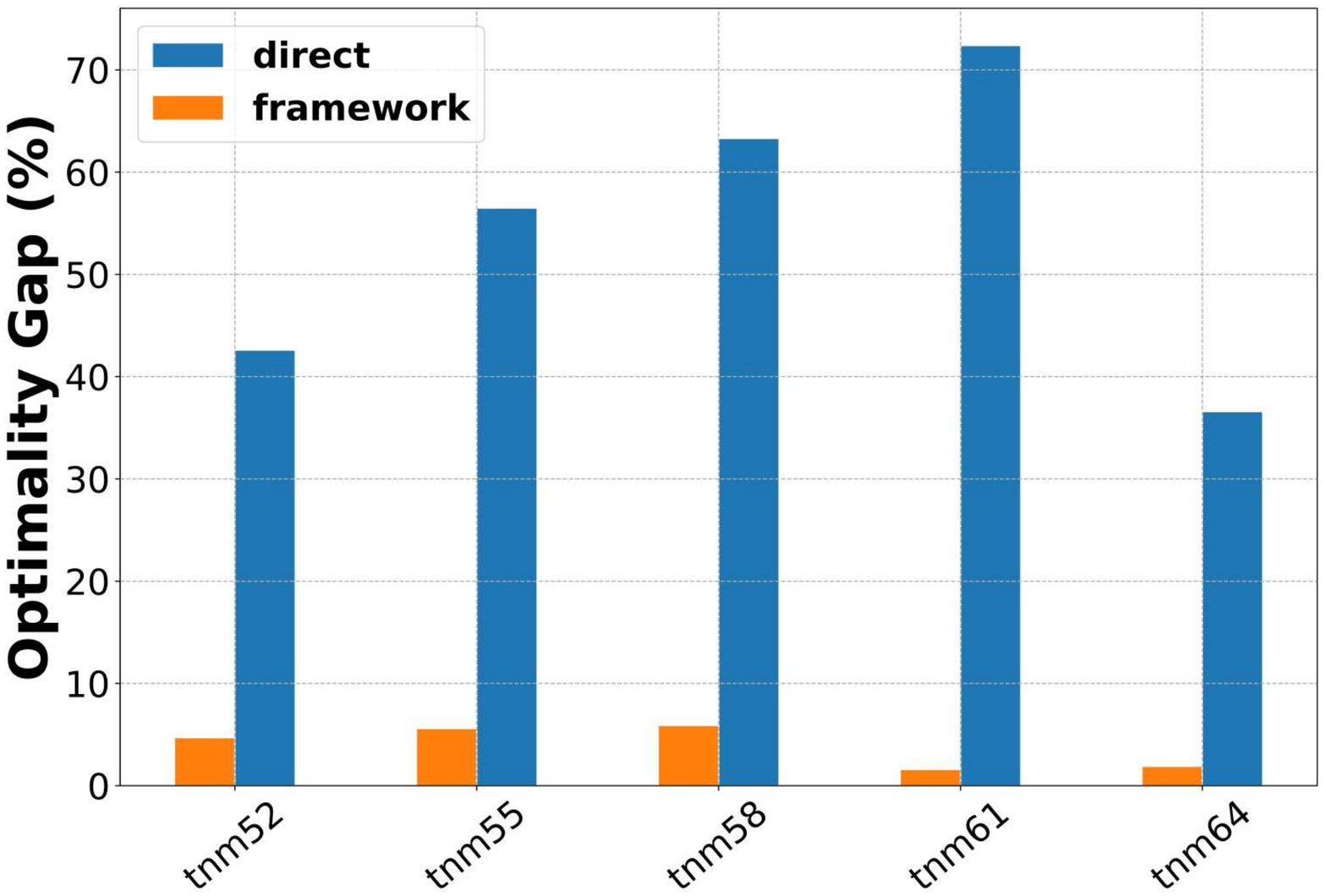}
  \caption{Tnm  Instances}
  \label{fig:sub1}
\end{subfigure}\qquad
\begin{subfigure}{0.45\textwidth}
  \centering
  \includegraphics[width=0.9\linewidth]{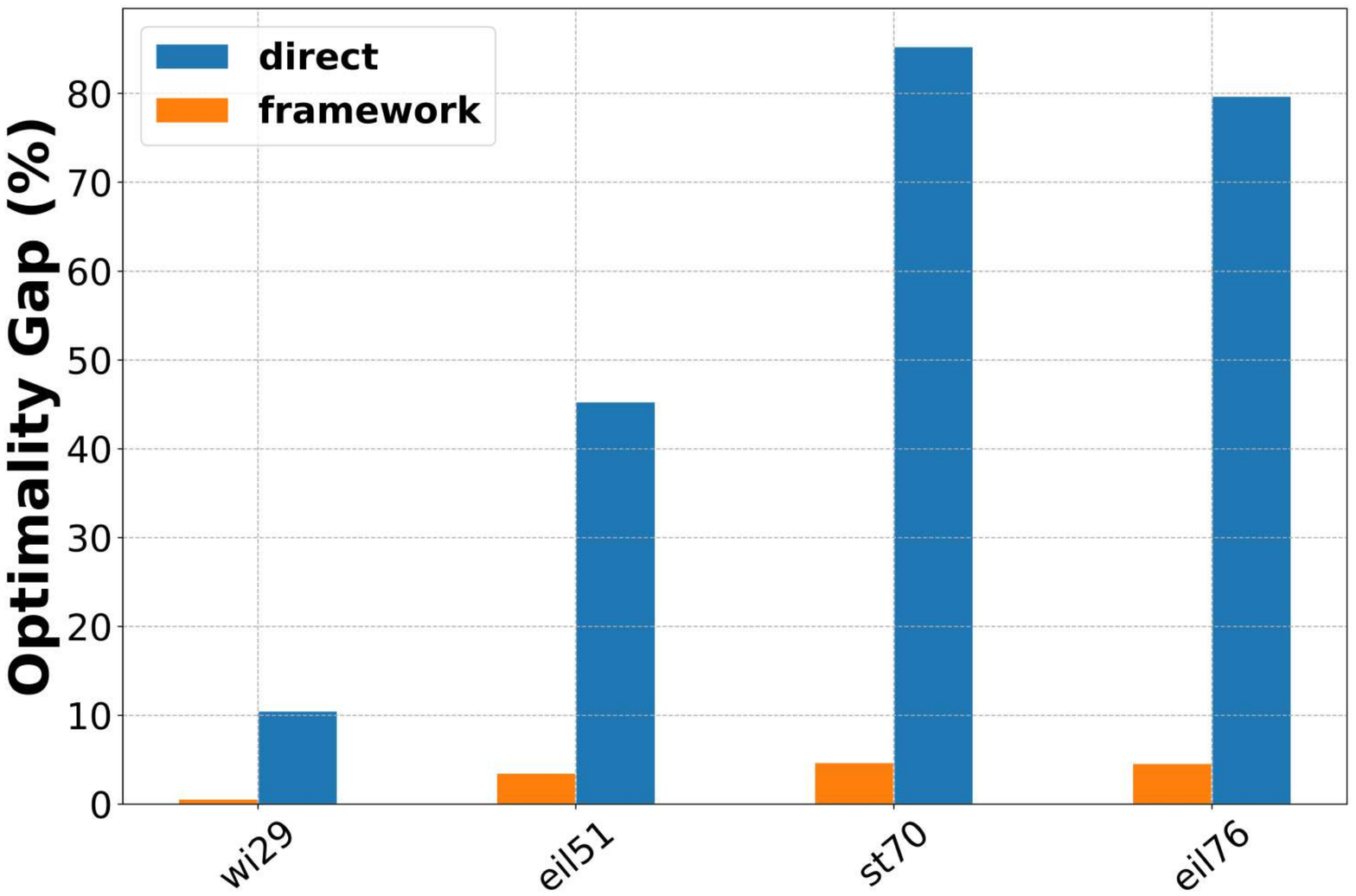}
  \caption{TSPLIB Instances}
  \label{fig:sub2}
\end{subfigure}
\caption{Comparison of performance if we use the maximum distance as the penalty parameter without the hybrid framework against using a trained MLP with the hybrid framework. }
\label{fig:directvs}
\end{figure}

\subsection{Comparison with Using an Exact Solver to Solve Sub-QUBO}

 In this paper, QS uses a meta-heuristic to solve small sub-QUBOs obtained from our decomposition framework.  It is interesting to study the performance if we use an exact solver such as CPLEX instead. 
 It is known that the required time for an exact solver can be exponentially long.  In order to compare the performance of these two solvers fairly,  we  consider two experiment settings.
 
 In the first setting, we solve each sub-QUBO model using CPLEX by setting the time limit to be $2$-minutes.  The result is as shown in Figure \ref{fig:comparecplexfirst}. The result shows that within the time limit,  $QS$ outperfoms the exact solver and the gap grows as the instance such increases.

 \begin{figure}
     \centering
     \includegraphics[width=7cm]{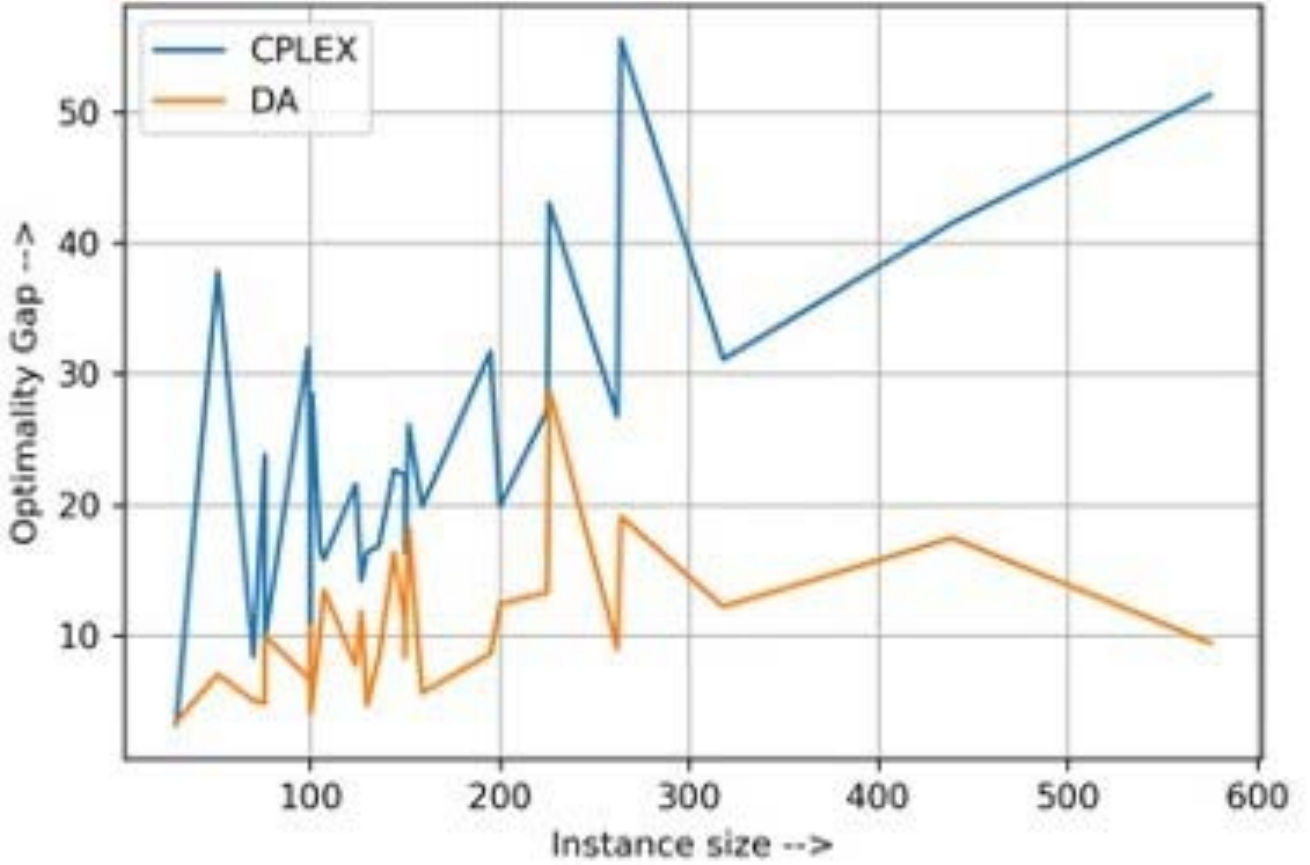}
     \caption{Comparison of run time CPLEX and QS without using optimality gap as an early termination condition but fixing run time.}
     \label{fig:comparecplexfirst}
 \end{figure}

In the second experiment setting, we aim to make the  optimality gap of the two solvers to be the same and examine the computational budget required to achieve the specified gap. For this purpose, after solving the E-TSP instances using QS, we note the optimality gap attained. The optimality gap can be derived, since the problem instances are obtained from a repository with known optimal solutions. We then set a maximum time limit of $5$ minutes in solving each sub-QUBO using CPLEX. Unlike the previous setting, another stopping condition that we use on CPLEX is that we set the required optimality gap for CPLEX to match the optimality gap obtained by QS. The result is as shown in Figure \ref{fig:comparecplex}. Again, QS attains a smaller optimality gap within a shorter run time.

\begin{figure}%
    \centering
    \subfloat[\centering Optimality Gap]{{\includegraphics[width=7cm]{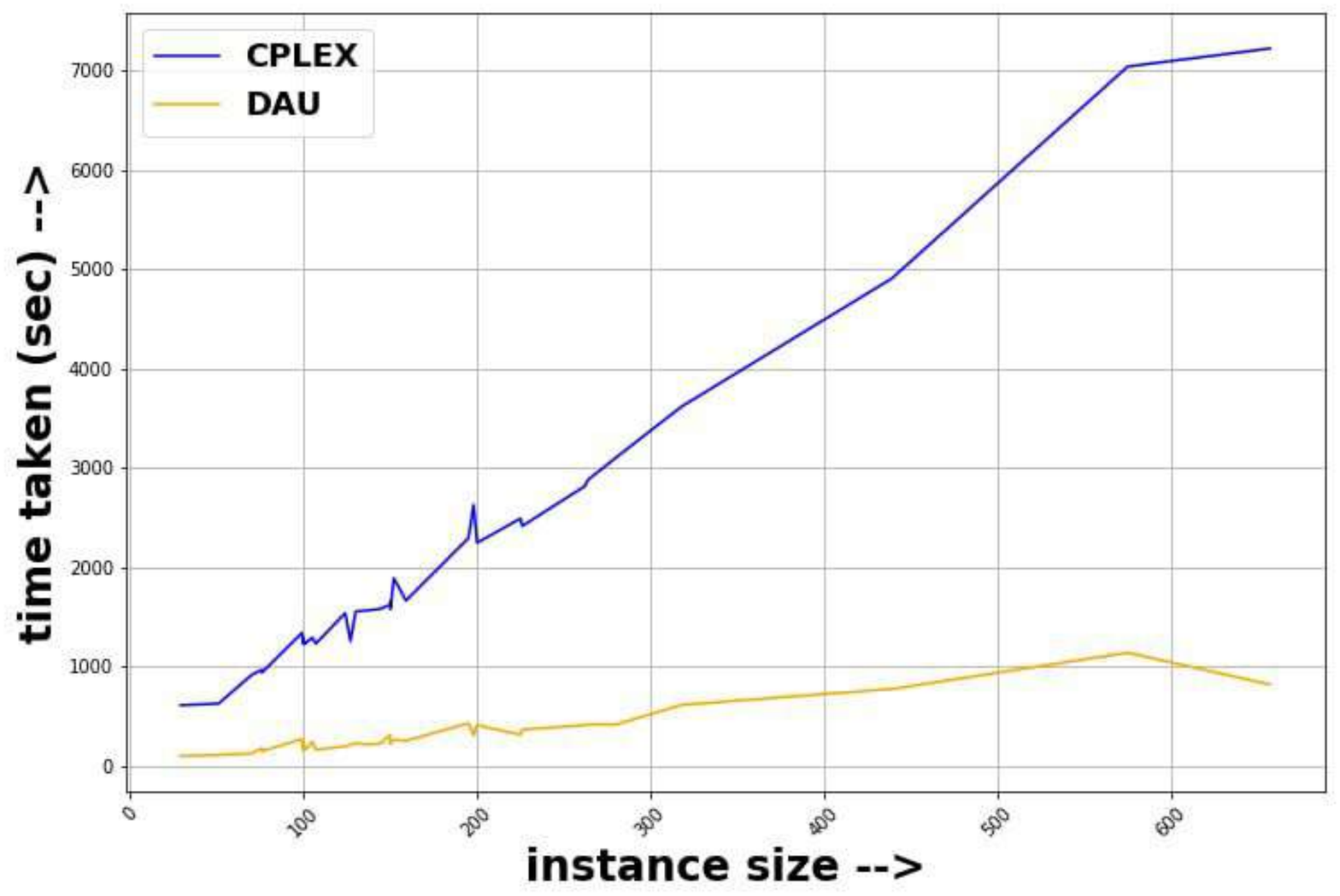} }}%
    \qquad
    \subfloat[\centering Run Time]{{\includegraphics[width=7cm]{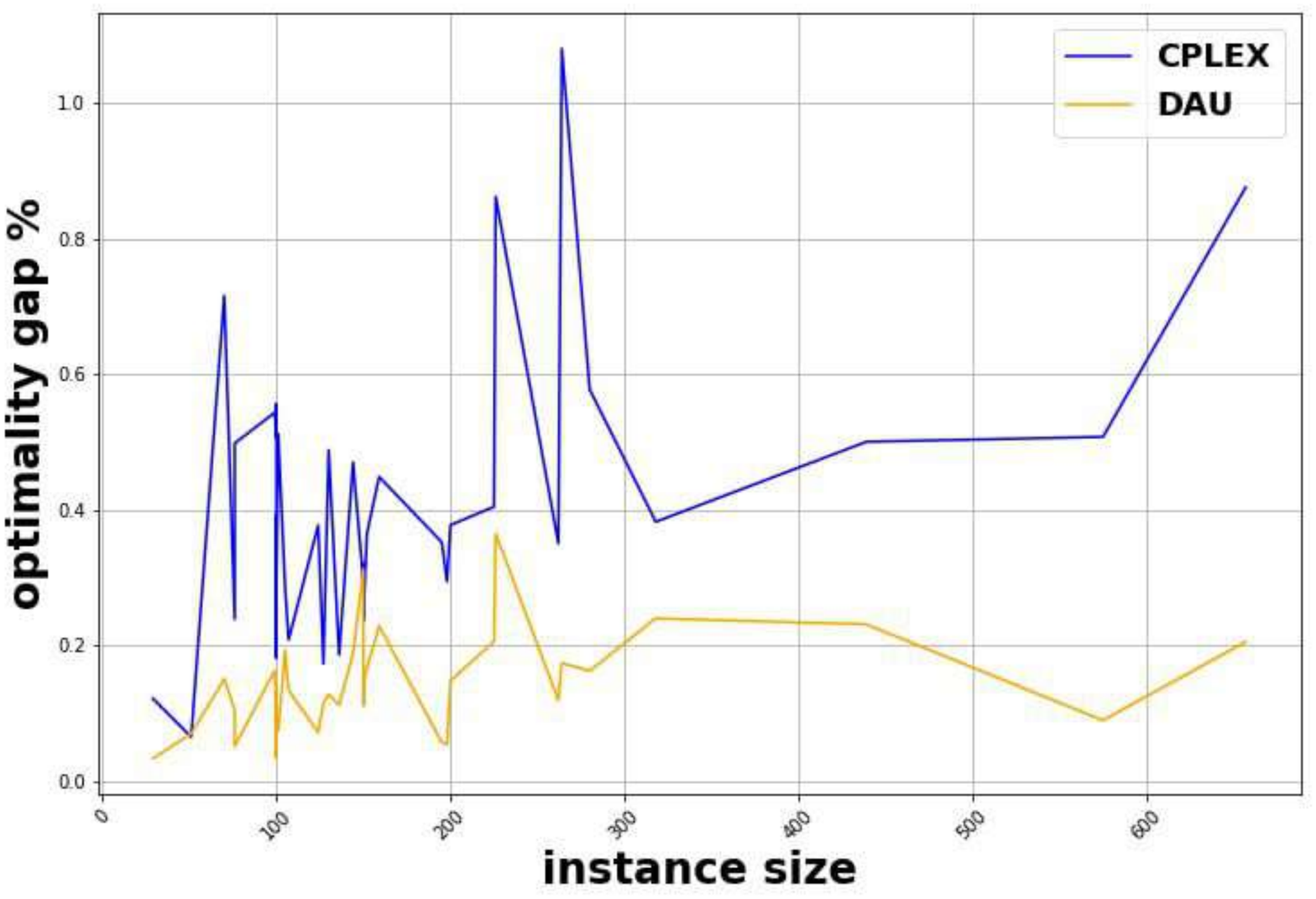}}}%
    \caption{Comparison of CPLEX and QS by setting CPLEX early termination condition to be  DA optimality gap.}%
    \label{fig:comparecplex}%
\end{figure}

In summary, we observe that QS performs better than CPLEX in solving the sub-QUBOs with either a fixed computational budget, or a fixed optimality gap. Arguably, with higher computational budget, CPLEX would outperform QS since it is an exact solver. 

\subsection{Results on Solving E-TSP}

\begin{figure}%
    \centering
    \subfloat[\centering Optimality Gap]{{\includegraphics[width=7cm]{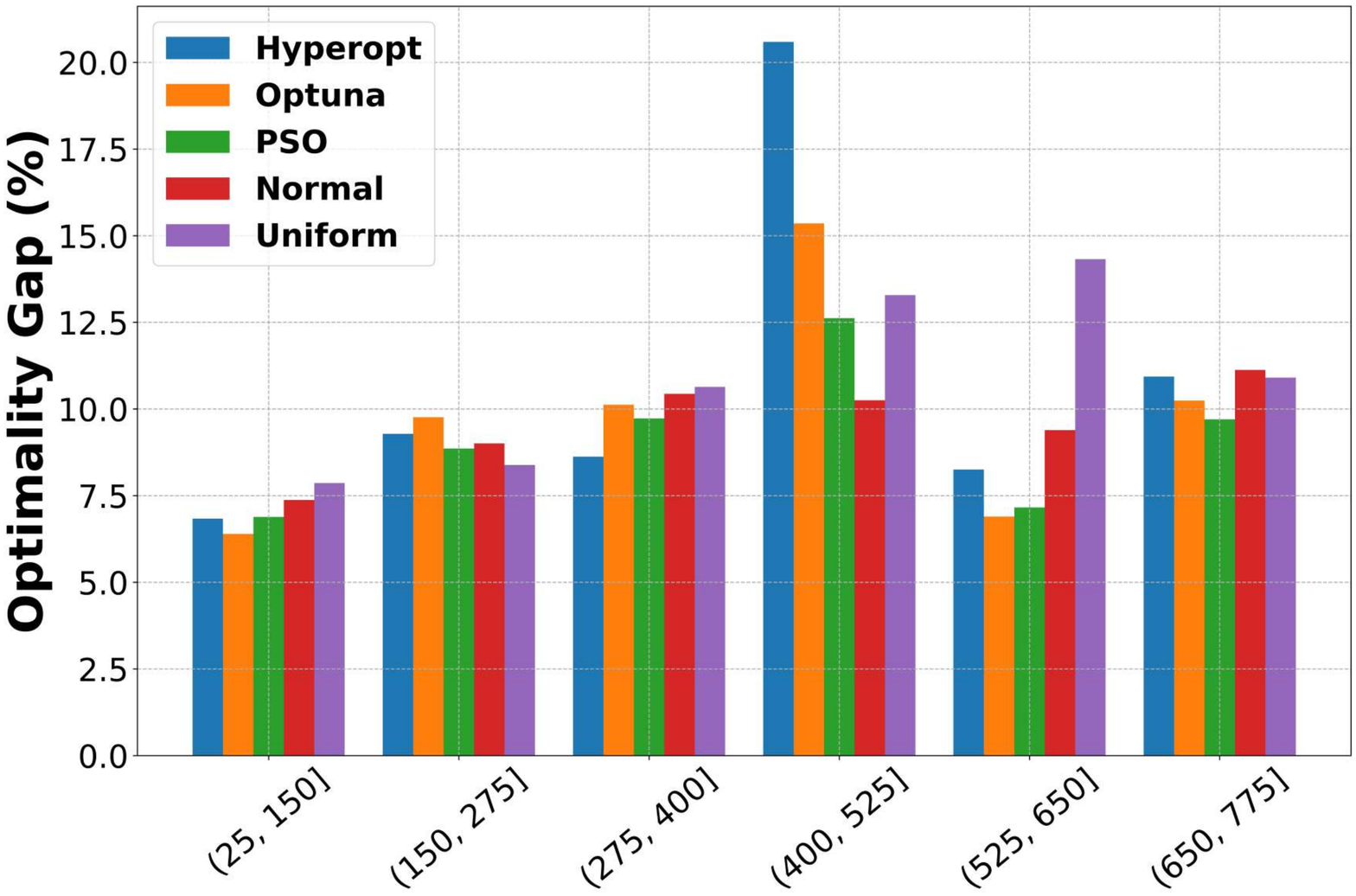} }}%
    \qquad
    \subfloat[\centering Tuning Time]{{\includegraphics[width=7cm]{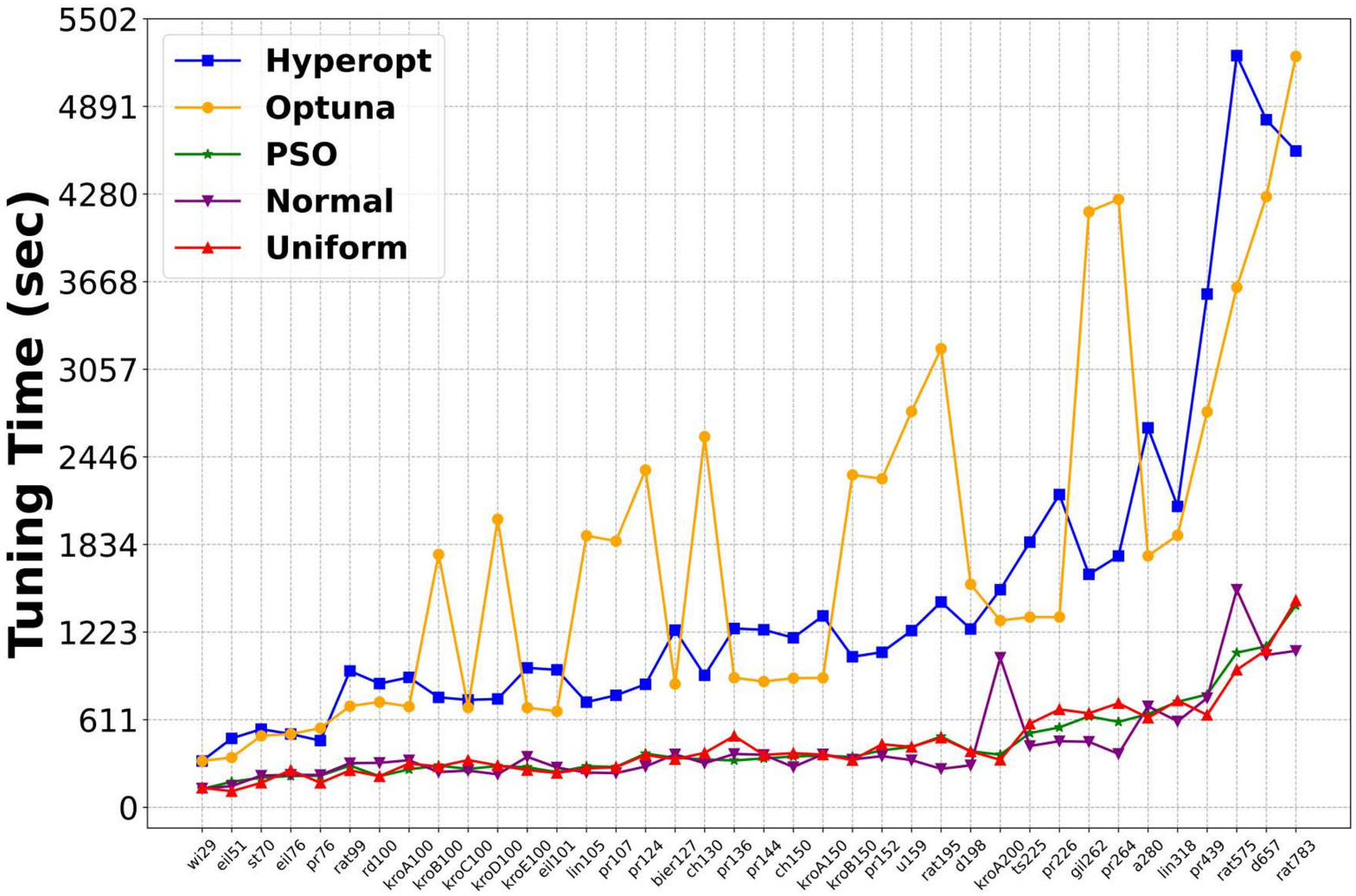}}}%
    \caption{Optimality gap and tuning time of TSP instances from TSPLIB}%
    \label{fig:tsplibplot}%
\end{figure}

\begin{figure}%
    \centering
    \subfloat[\centering Optimality Gap]{{\includegraphics[width=7cm]{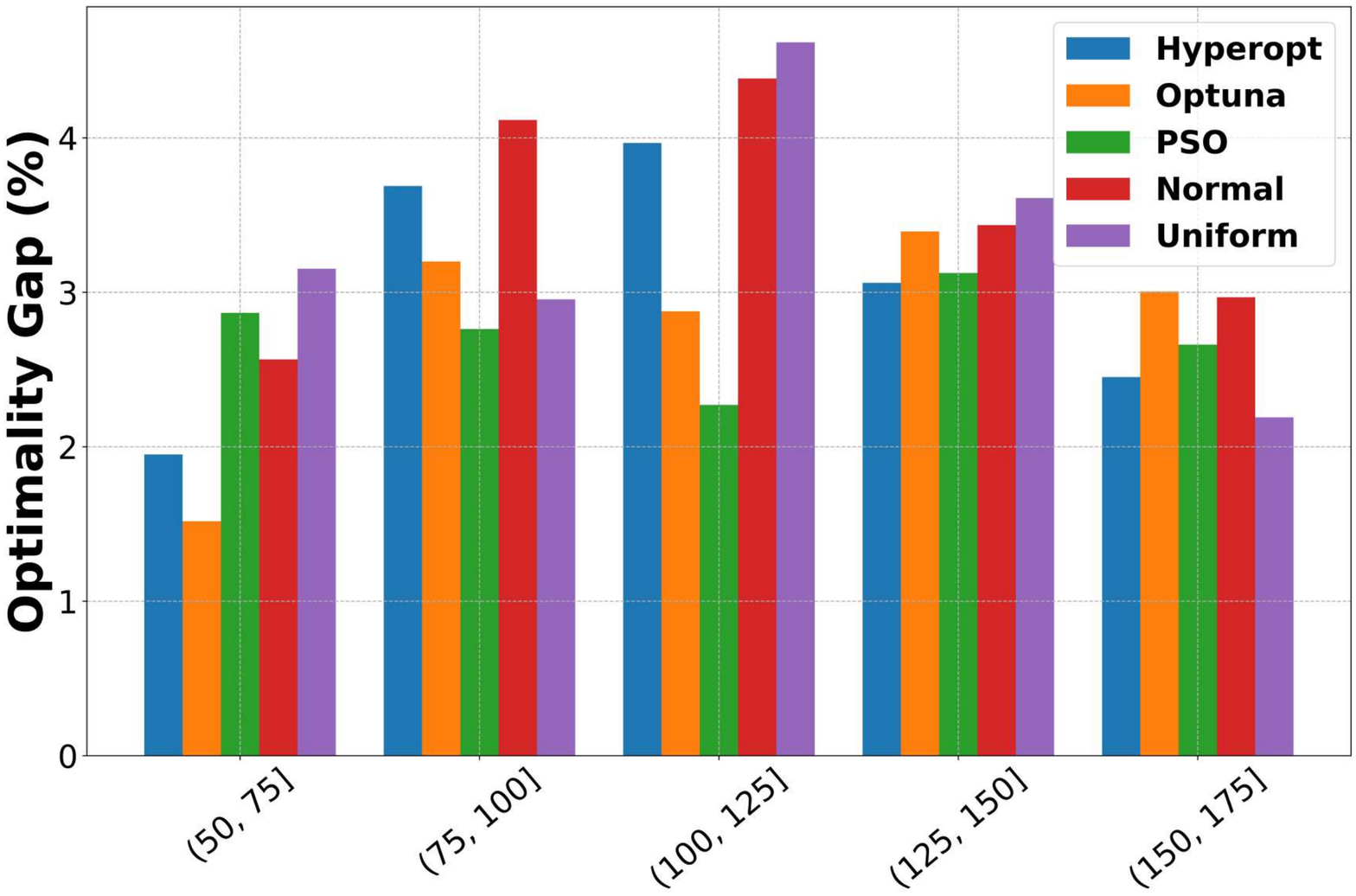} }}%
    \qquad
    \subfloat[\centering Tuning Time]{{\includegraphics[width=7cm]{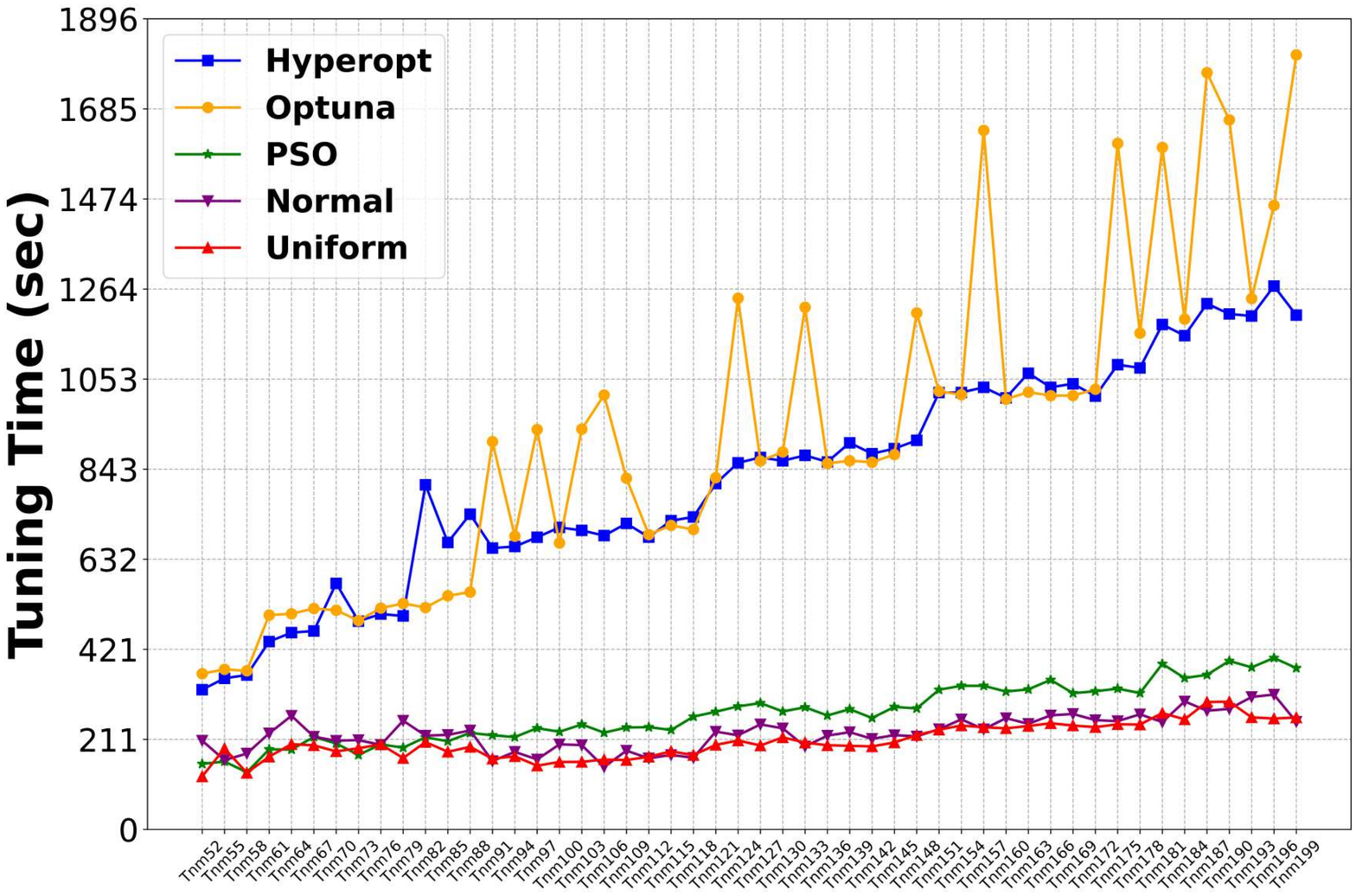}}}%
    \caption{Optimality gap and tuning time of TSP instances from \cite{hougardy2018hard}}%
    \label{fig:tnmplot}%
\end{figure}

We perform online parameter tuning on the TSP instances of size up to $800$ cities including some instances that are difficult for Concorde to solve. The data are obtained from TSPLIB and \cite{hougardy2018hard} respectively. 
Figure \ref{fig:tsplibplot} and Figure \ref{fig:tnmplot}  show the optimality gap and the tuning time when we use the five online parameter tuning schemes that we discussed  earlier. We grouped the instances according to their size in reporting the optimality gap.  The optimality gap can be obtained as our test problems are obtained from a repository where the optimal values of the instances have been reported. We observe that the optimality gap is comparable with each other for most instances, even though on average, Optuna performs the best. Interesting, on the hard E-TSP problem instances, we obtained optimality gap of at most $5\%$. 

For this application, it seems that tuning the penalty parameter uniformly with ratio between $0.5$ to $1$ of the maximum weight length suffices to obtain good results experimentally.  However, this does not imply that careful  tuning is not required. In general, a helpful principle we read from our experiments is that we need to first narrow the promising search space, from which we can perform a simple scheme such as uniform search to find the most appropriate values.

\subsection{Results on Solving FSP}

For FSP, the dataset is drawn from \cite{vallada2015new}. Due to the calculation of the distance between two jobs in \cite{stinson1982heuristic} involving manipulating the processing time of the two over all the machines, this approximation method cannot perform well for those instances with large number of machines. Though the dataset contains the instances with $n \in [10, 800]$ and $m \in [5, 60]$, we chose total 120 instances with $n \leq 60$ and $m \leq  10$ for our experiments.

We show the performance of various parameter tuning scheme in terms of optimality gap against the best known solution and runtime. We aggregate the optimality gap and tuning time given the size of the problem instance. Instances with same number of jobs and machines are grouped together.

We can see from  Figure \ref{fig:fsp} that the optimality gap is comparable with each other regardless the size of the instances. We can also see that for the instances with $m=10$, the optimality gap is significantly larger than those instances that contains only $5$ machines. On the other hand, the tuning time of different schemes are all increasing given the number of the jobs accordingly.

Again, the statistical sampling approach, leveraged over the parallelizability of the QS, is the fastest among all the schemes.

\begin{figure}%
    \centering
    \subfloat[\centering Optimality Gap]{{\includegraphics[width=7cm]{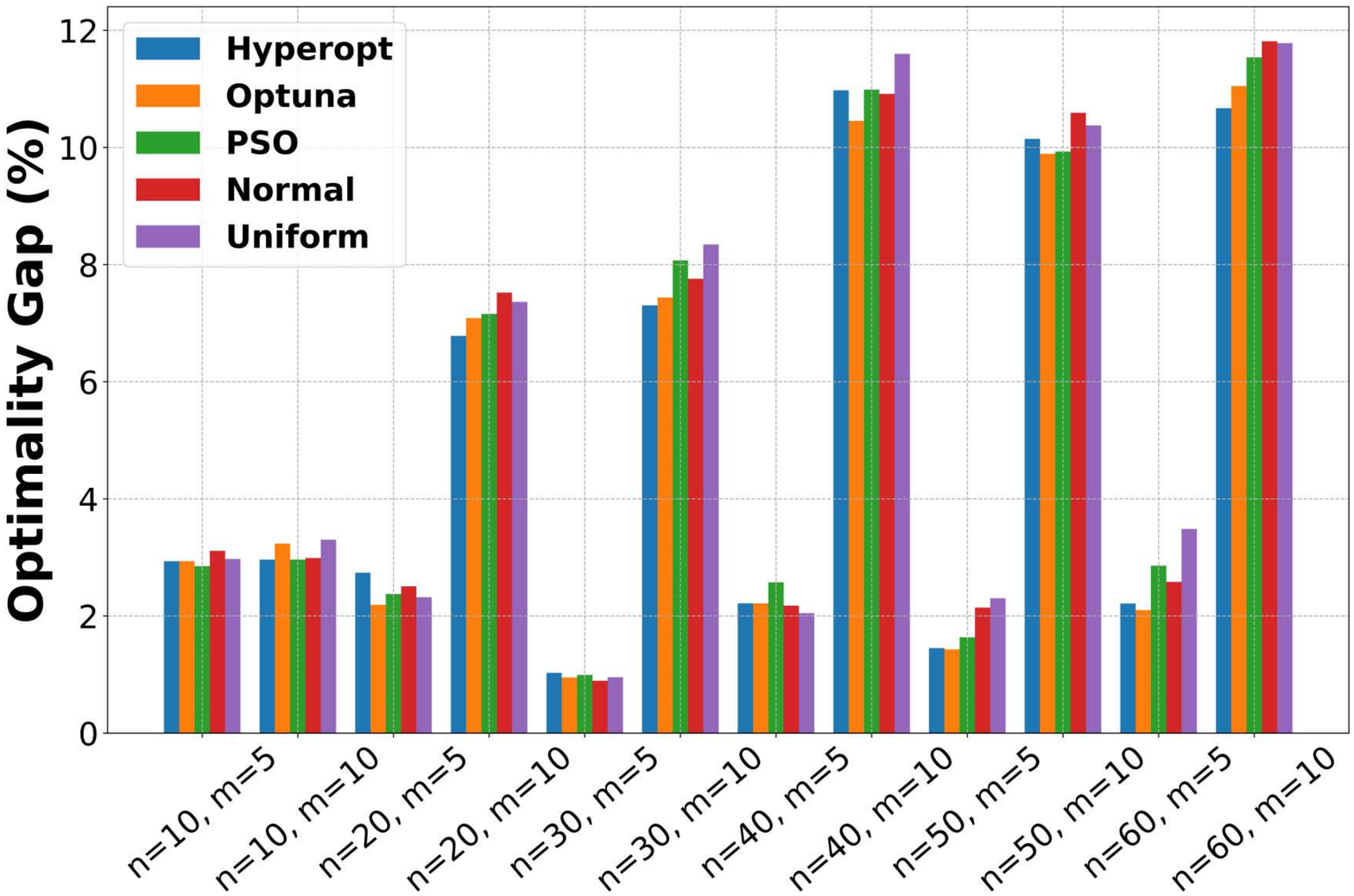} }}%
    \qquad
    \subfloat[\centering Tuning Time]{{\includegraphics[width=7cm]{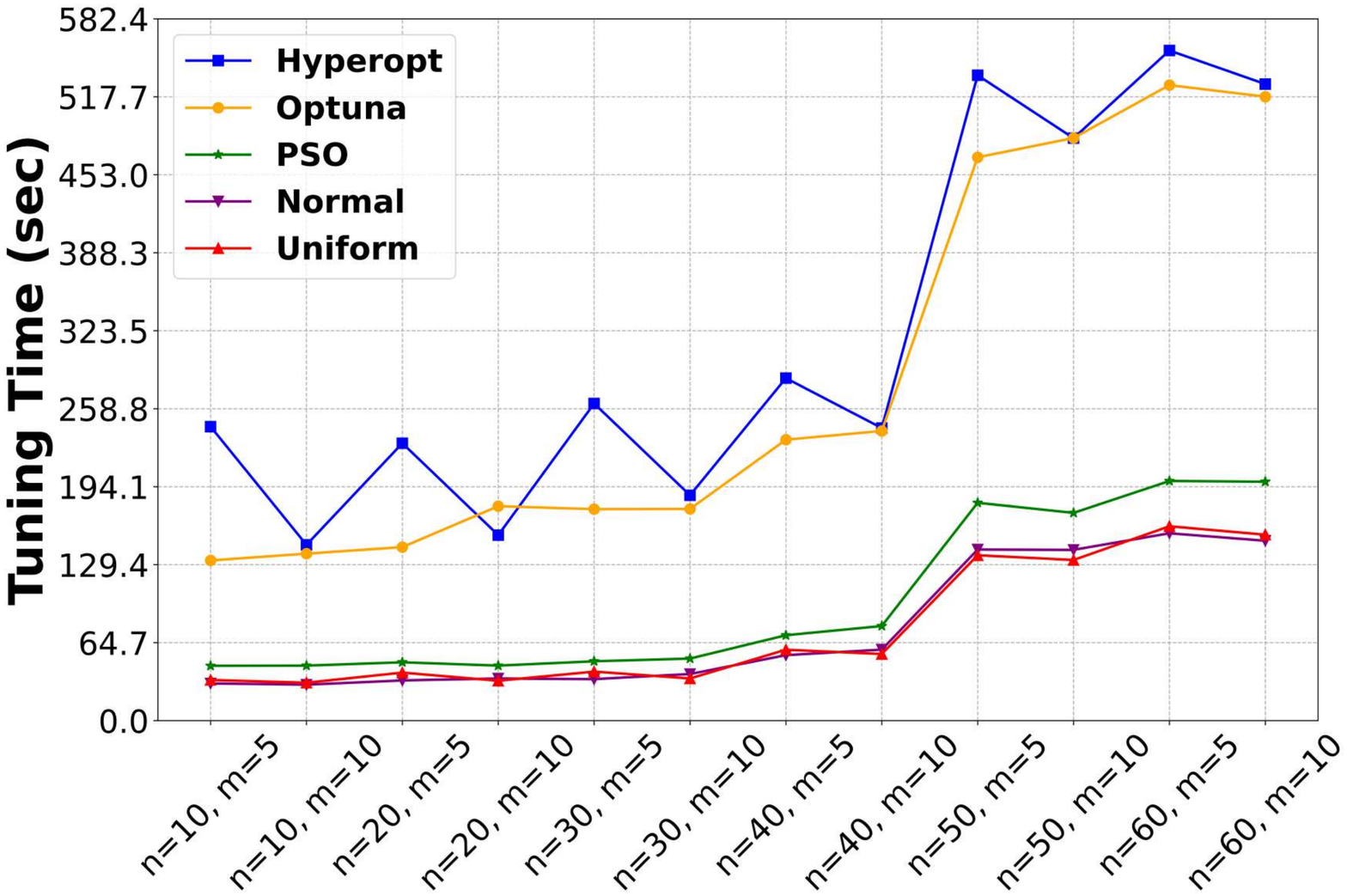}}}%
    \caption{Optimality gap and tuning time of FSP instances from \cite{vallada2015new}}%
    \label{fig:fsp}%
\end{figure}


\section{Conclusion}

In this paper, we proposed a framework aimed at exploiting a heuristic QUBO solver to solve large-scale permutation-based combinatorial optimization problems. This approach yields very small optimality gaps on E-TSP and FSP  benchmark instances as our experimental results showed. 

We have also illustrated that our framework performs better than applying a large problem to the QS directly even if the problem size can fit into the QS input size. This is due to as problem size increases, quality of QS tend to be worse due to multiple reasons such as  large search space.  We also numerically illustrated that QS can solve problem much better than an exact solver especially under the setting of limited time budget.



We conclude with a reflective remark. Our framework applies to other heuristic QUBO solvers and even quantum annealers. The framework would apply directly to a quantum hardware. The main differences are due to the absence of all-to-all connections, there could be additional tasks involving embedding the sub-problems to the quantum hardware.  When quantum annealers become more commercially viable, this framework could be utilized by a quantum annealer for solving large-scale combinatorial optimization problems that surpass computationally what exact solvers like CPLEX and Gurobi are capable of doing today.  



\newpage

\bibliographystyle{IEEEtran}
\bibliography{mybib.bib}
\end{document}